\documentclass[11pt, a4paper]{amsart}
\usepackage{amsfonts}
\usepackage{amscd,amsmath}
\usepackage{amssymb}
\usepackage{setspace}
\usepackage{fullpage}
\usepackage{amsfonts}
\usepackage{amscd,amsmath}
\usepackage{amssymb}
\usepackage{tikz-cd}
\usepackage{mathtools}
\newtheorem{theorem}{Theorem}[section]
\newtheorem{lemma}[theorem]{Lemma}

\newtheorem{proposition}[theorem]{Proposition}
\newtheorem{definition}[theorem]{Definition}

\newtheorem*{TI*}{Topological Interpretation}

\theoremstyle{remark}
\newtheorem{remark}{Remark}[section]
\newtheorem*{rem*}{Remark}

\newcommand{\comment}[1]{}
\numberwithin{equation}{section}
\DeclareMathOperator{\Um}{Um}
\DeclareMathOperator{\fr}{first}
\DeclareMathOperator{\co}{column}
\DeclareMathOperator{\of}{of}

\begin{document}
\title{Quillen Suslin theory for algebraic fundamental group}
\author{RAJA SRIDHARAN $^1$, Sumit Kumar Upadhyay $^2$  AND SUNIL KUMAR YADAV$^3$\vspace{.4cm}\\
{$^{1}$School Of Mathematics,\\ Tata Institute of fundamental Research,\\Colaba,  Mumbai, India \vspace{.3cm}\\ $^{2}$Department of Applied Sciences,\\ Indian Institute of Information Technology Allahabad, \\Prayagraj, U. P., India\\} \vspace{.3cm}\\ $^{3}$Department of Mathematics, \\ Lady Sri Ram College of Women, \\University of Delhi, India}
\thanks{$^1$sraja@math.tifr.res.in, $^2$upadhyaysumit365@gmail.com, $^3$skymath.bhu@gmail.com}

\begin{abstract}
In this paper, we attempt to develop the Quillen Suslin theory for the algebraic fundamental group of a ring. We give a surjective group homomorphism from the algebraic fundamental group of the field of the real numbers to the group of integers. At the end of the paper, we also propose some problems related to the algebraic fundamental group of some particular type of rings.
\end{abstract}
\maketitle
\section{Introduction}
There is no general method to solve the  basic topological problem ``whether two given topological spaces are homeomorphic" but there exist techniques that apply in particular cases. If there is some topological property that holds for one space but not for the other, then the spaces cannot be homeomorphic. Later, a more general idea came into the picture which involves a certain group that is called the fundamental group of the space. The task of computing some fundamental groups are not trivial.  
The notion of covering space is one of the most useful tools for computing some fundamental groups. In \cite{RSS}, for a commutative ring $R$ with identity, we introduced the known groups $\pi_1 (SL_2(R))$ and $\Gamma(R)$ to prove an algebraic analogue of the Mayer Vietoris sequence. By the definition, the group $\pi_1 (SL_2(R))$ looks like an algebraic analogue of the fundamental group $\pi_1 (SL_2(\mathbb{R}), I_2)$ of the space $SL_2(\mathbb{R})$ via polynomial paths, where $\mathbb{R}$ is the field of real numbers. In this paper, we call the group $\pi_1 (SL_2(R))$ as the algebraic fundamental group of $R$. We do not know the accurate history of the ideas of these groups. However, we are interested in explaining how these ideas arose to us. We apologize for any inaccuracies.

Let $X$ and $Y$ be topological spaces, and $\text{Cont}(X, Y)$ be the set of all continuous maps from $X$ to $Y$. To understand the meaning of the group $\pi_1 (SL_2(R))$, one consider the following idea for a topological space $X$, the set $\text{Cont}(X, \mathbb{R})$ is a ring. Also $\text{Cont}(X, \mathbb{R}^n)\cong \oplus_n \text{Cont}(X, \mathbb{R})$. Similarly, the set $\text{Cont}(X, \pi_1 (SL_2(\mathbb{R})))$ is $\pi_1(SL_2($ $\text{Cont}(X, \mathbb{R})))$. Since $\pi_1 (SL_2(\mathbb{R}))$ is $\mathbb{Z}$, $\pi_1(SL_2($ $\text{Cont}(X, \mathbb{R})))$ is $\text{Cont}(X, \mathbb{Z})$. Therefore, the group $\pi_1 (SL_2(R))$ is related to the connectedness. This is probably well known and possibly one of the motivations for considering the group $\pi_1 (SL_2(R))$.  In this article, we will try to develop the Quillen Suslin for the group $\pi_1 (SL_2(R))$, for a ring $R$.

Now, we give a motivation behind the definition of the group $\pi_1 (SL_2(R))$ due to Nori using the concept of construction of universal covering space for a topological space (for universal covering space we refer \cite{JR}). 

Consider the field of complex numbers $\mathbb{C}$ and the exponential map $\exp : \mathbb{C} \rightarrow \mathbb{C}^*$, where $\mathbb{C}^* = \mathbb{C} \setminus \{0\}$. Then we have the following short exact sequence of groups
$$0 \rightarrow \mathbb{Z} \rightarrow \mathbb{C} \overset{\exp}\rightarrow \mathbb{C}^* \rightarrow 0 \hspace{2cm} (1)$$
Since every function from $X$ to $\mathbb{C}^*$ need not have logarithm, we have the
following exact sequence of sheaves
$$1 \rightarrow \text{Cont}(X, \mathbb{Z}) \rightarrow \text{Cont}(X, \mathbb{C}) \overset{\exp}\rightarrow \text{Cont}(X, \mathbb{C}^* ) \hspace{2cm} (2)$$
 because $\exp:\text{Cont}(X, \mathbb{C}) \rightarrow \text{Cont}(X, \mathbb{C}^* )$ need not be surjective in general but it is locally surjective.

Now, consider $\mathbb{R}^2\setminus \{(0, 0)\}$ instead of $\mathbb{C}^*$ and write a sequence similar to the sequences $(1)$ and $(2)$ by using the idea of construction of the universal covering space of a topological space. Let $(1, 0)\in \mathbb{R}^2\setminus \{(0, 0)\}$ and $\mathcal{P} = \{ \alpha: I \to \mathbb{R}^2\setminus \{(0, 0)\} \mid \alpha(0) = (1, 0) \}$, where $I$ is the unit interval $[0, 1]$. Define an equivalence relation $\sim$ in $\mathcal{P}$ as follows:

$\alpha \sim \beta$ if $\alpha(1) = \beta(1)$ and there exists $\gamma: I\times I  \to B$ such that 
\begin{center}
$\gamma(t, 0) = \alpha(t)$, $\gamma(t, 1) = \beta(t)$\\
$\gamma(0, s) = (1, 0)$, $\gamma(1, s) = \alpha(1) =\beta(1)$
\end{center}
Let $E$ denote the set of all equivalence classes, that is, $E = \{[\alpha] \mid  \alpha \in \mathcal{P} \}$. Define a map $p : E \rightarrow \mathbb{R}^2\setminus \{(0, 0)\}$ by $p([\alpha]) = \alpha(1)$. Since $\mathbb{R}^2\setminus \{(0, 0)\}$ is path connected, $p$ is surjective. Now, $p^{-1}((1, 0)) = \{[\alpha]  \in E \mid  \alpha(1) = (1, 0) \} = \{[\alpha]  \in E \mid  \alpha(0) = \alpha(1) = (1, 0) \}$ which is the fundamental group $\pi_1 (\mathbb{R}^2\setminus \{(0, 0)\})$. Thus we have the following exact sequence of groups
$$0 \rightarrow \pi_1 (\mathbb{R}^2\setminus \{(0, 0)\}) \rightarrow E \overset{p}\rightarrow \mathbb{R}^2\setminus \{(0, 0)\} \rightarrow 0 \hspace{2cm} (3)$$

Let $R$ be a ring and $Um_2 (R)$ be the set of all unimodular rows of length $2$ (an element $(a, b)$ of $R^2$ is unimodular if $\langle a, b \rangle = A$). If $R$ is the field of real numbers $\mathbb{R}$, then $Um_2 (\mathbb{R})$ is  $\mathbb{R}^2\setminus \{(0, 0)\}$ and hence $\pi_1 (\mathbb{R}^2\setminus \{(0, 0)\})= \pi_1 (Um_2 (\mathbb{R}))$. In fact, we can rewrite the sequence $(3)$ as  
$$0 \rightarrow \pi_1 (Um_2 (\mathbb{R})) \rightarrow E \overset{p}\rightarrow Um_2 (\mathbb{R}) \rightarrow 0 \hspace{2cm} (3')$$
 
Now, suppose $R$ is the coordinate ring of a real affine variety $X$ with real points $X(\mathbb{R})$. Then from the sequence $(3')$, we have \\
$1 \rightarrow \text{Cont}(X(\mathbb{R}), \pi_1 (Um_2 (\mathbb{R})) \rightarrow \text{Cont}(X(\mathbb{R}), E) \overset{p}\rightarrow \text{Cont}(X(\mathbb{R}), Um_2 (\mathbb{R})) \hspace{1cm} (4)$

Since every unimodular row gives a continuous map from $X(\mathbb{R})$ to $\mathbb{R}^2 \setminus \{(0, 0)\}$, we can say that  $Um_2 (R)$ is the algebraic analogue of $\text{Cont}(X(\mathbb{R}), \mathbb{R}^2 \setminus \{(0, 0)\})$. Therefore, the sequence $(4)$ is the following\\
$$1 \rightarrow \pi_1 (Um_2 (R)) \rightarrow E' \overset{p}\rightarrow Um_2 (R)\hspace{2cm}  (4')$$
where $\pi_1 (Um_2 (R))$ is the algebraic analogue of $\text{Cont}(X(\mathbb{R}), \pi_1 (\mathbb{R}^2\setminus \{(0, 0)\}))$ and $E'$ is the algebraic analogue of $\text{Cont}(X(\mathbb{R}), E)$. The above is only a heuristic remark. A. Bak and A. S. Garge \cite{BG} developed the notion of the group $\pi_1 (Um_n (R))$, for a ring $R$,  using the theory of global actions which is applicable in many other contexts. We now outline Nori's approach which involves considering the group $SL_2 (R)$ instead of $Um_2(R)$

Consider the space $B= SL_2 (\mathbb{R})$ and $\mathcal{P} = \{ \alpha(T) \in SL_2 (\mathbb{R}[T]) \mid \alpha(0) = I_2\}$. Define an equivalence relation $\sim$ in $\mathcal{P}$ as follows:

$\alpha(T) \sim \beta(T)$ if $\alpha(1) = \beta(1)$ and there exists $\gamma(T, S) \in SL_2 (\mathbb{R}[T, S])$ such that 
\begin{center}
$\gamma(T, 0) = \alpha(T)$, $\gamma(T, 1) = \beta(T)$\\
$\gamma(0, S) = I_2$, $\gamma(1, S) = \alpha(1) =\beta(1)$
\end{center}
Let $E$ denote the set of all equivalence classes, that is, $E = \{[\alpha(T)] \mid  \alpha(T) \in \mathcal{P} \}$. Define a map $p : E \rightarrow B$ by $p([\alpha(T)]) = \alpha(1)$. 
Now, $p^{-1}(I_2) = \{[\alpha(T)]  \in E \mid  \alpha(1) = I_2 \} = \{[\alpha(T)]  \in E \mid  \alpha(0) = \alpha(1) = I_2 \}$ which is the fundamental group $\pi_1(SL_2(\mathbb{R}))$ via the polynomial paths. Thus we have the following exact sequence of groups
$$0 \rightarrow \pi_1(SL_2(\mathbb{R})) \rightarrow E \overset{p}\rightarrow SL_2 (\mathbb{R}) \rightarrow 0\hspace{2cm}  (5)$$

Let $R$ be the coordinate ring of a real affine variety $X$ with real points $X(\mathbb{R})$. Then from the sequence $(5)$, we have \\
$$1 \rightarrow \text{Cont}(X(\mathbb{R}), \pi_1(SL_2(\mathbb{R}))) \rightarrow \text{Cont}(X(\mathbb{R}), E) \overset{p}\rightarrow \text{Cont}(X(\mathbb{R}), SL_2 (\mathbb{R})) \hspace{1cm} (5')$$

Since every element of $SL_2(R)$ gives a continuous map from $X(\mathbb{R})$ to $SL_2(\mathbb{R})$, we can say that  $SL_2 (R)$ is the algebraic analogue of $\text{Cont}(X(\mathbb{R}), SL_2(\mathbb{R}))$ and $\pi_1(SL_2(R))$ is the algebraic analogue of $\text{Cont}(X(\mathbb{R}), \pi_1(SL_2(\mathbb{R})))$. Thus the sequence $(5')$ is 
$$1 \rightarrow \pi_1(SL_2(R)) \rightarrow G'(R) \overset{p}\rightarrow SL_2(R) \hspace{1cm} (5'')$$
In the sequence $(5'')$,   $G'(R) = \{[\alpha(T)] \mid  \alpha(T) \in SL_2(R[T]) ~\text{with}~ \alpha(0) = I_2 \}$ where $[\alpha(T)]$ denotes the equivalence class of $\alpha(T) \in SL_2(R[T])$ with $\alpha(0)= I_2$ and equivalence relation $\sim$ is defined as follows:

$\alpha(T) \sim \beta(T)$ if $\alpha(1) = \beta(1)$ and there exists $\gamma(T, S) \in SL_2 (R[T, S])$ such that 
\begin{center}
$\gamma(T, 0) = \alpha(T)$, $\gamma(T, 1) = \beta(T)$\\
$\gamma(0, S) = I_2$, $\gamma(1, S) = \alpha(1) =\beta(1)$.
\end{center} Since  the actual fundamental group of  $SL_{2}(\mathbb{R})$ is isomorphic to $\mathbb{Z}$, we can say that $\pi_1 (SL_{2}(R))$ is the set of all continuous maps from $X(\mathbb{R})$ to $\mathbb{Z}$. That is $\pi_1 (SL_{2}(R))$ can be thought of as an algebraic analogue of the topological group $H^0 (X(\mathbb{R}), \mathbb{Z})$. Nori outlined these above ideas to the first author in the early nineties. Nori wanted to know if one could  use the connecting homomorphism $H^1(-, SL_{2}(R))$ to $H^2(-, \pi_1 (SL_{2}(R)))$ associated to the exact sequence $(5'')$ to define Euler class of rank two algebraic vector bundles over two dimensional ring.

Let $U_1$ and $U_2$ be two open sets of a topological space $X$. Then we have the following Mayer-Vietoris sequence  

$H^0(U_1\cup U_2, \mathbb{Z})\rightarrow H^0(U_1, \mathbb{Z})\oplus H^0(U_2, \mathbb{Z})\rightarrow H^0(U_1\cap U_2, \mathbb{Z}){\overset{\phi}\rightarrow} H^1(U_1\cup U_2, \mathbb{Z})\rightarrow 
H^1(U_1, \mathbb{Z})\oplus H^1(U_2, \mathbb{Z})\rightarrow H^1(U_1\cap U_2, \mathbb{Z})\hspace{2cm} (6)$

Regarding the details of the groups $H^0(Y, \mathbb{Z})$ and $H^1(Y, \mathbb{Z})$ for a topological space $Y$, and the Mayer-Vietoris sequence we refer \cite{CTC}, where $H^1(Y, \mathbb{Z})$ is the group of homotopy classes of continuous maps from $Y$ to $\mathbb{C^*}$. In the sequence $(6)$, the connecting map $\phi: H^0(U_1\cap U_2, \mathbb{Z}) \rightarrow H^1(U_1\cup U_2, \mathbb{Z})$ is defined as follows: let $f \in H^0(U_1\cap U_2, \mathbb{Z})$. Since $f: U_1\cap U_2 \rightarrow \mathbb{Z}$, we can think it as a map from $ U_1\cap U_2 \rightarrow \mathbb{C}$. Then $f = {f_1}_{|U_1\cap U_2} - {f_2}_{|U_1\cap U_2}$, for some $f_i : U_i \rightarrow \mathbb{C}$, $i=1, 2$. Now, we define $$\phi(f) = \begin{cases}
\exp{(2\pi i f_1 (x))}, ~\text{if}~ x\in U_1\\
\exp{(2\pi i f_2 (x))}, ~\text{if}~ x\in U_2
\end{cases}$$
It is clear that $\phi(f): U_1\cup U_2 \rightarrow \mathbb{C^*}$, that is, $\phi(f)\in H^1(U_1\cup U_2, \mathbb{Z})$.
Let $R$ be a ring and $(a, b) \in Um_2(R)$. Suppose $R$ is the coordinate ring of a real affine variety $X$ with real points $X(\mathbb{R})$. Then the rings $R_a$ and $R_b$ are the coordinate rings of some open subsets $U_1$ and $U_2$ of $X(\mathbb{R})$, respectively. Also, $R_{ab}$ is the coordinate ring of open subset $U_1\cap U_2$. From the above discussion, we  have seen that $\pi_1 (SL_{2}(-))$ is an algebraic analogue of $H^0(-, \mathbb{Z})$ and in \cite{RSS}, we mentioned that $\Gamma(-)$ is an algebraic analogue of  $H^1(-, \mathbb{Z})$. Hence, from the sequence $(6)$, we have 

$\pi_1 (SL_{2}(R))\rightarrow \pi_1 (SL_{2}(R_a))\oplus \pi_1 (SL_{2}(R_b))\rightarrow \pi_1 (SL_{2}(R_{ab})){\overset{\delta}\rightarrow} \Gamma(R)\rightarrow \Gamma(R_a)\oplus \Gamma(R_b)\rightarrow \Gamma(R_{ab})\hspace{2cm} (6')$

The motivation for defining the connecting map $\delta: \pi_1 (SL_{2}(R_{ab}))\rightarrow \Gamma(R)$ using  $\phi$ is as follows. 
Let $[\beta(T)] \in \pi_1 (SL_{2}(R_{ab}))$. Then from the above discussion we can think $[\beta(T)]$ as a map from $U_1\cap U_2$ to $\mathbb{Z}$. From the sequence $(5'')$, we have 
$$0 \rightarrow \pi_1(SL_2(R_{ab})) \rightarrow G'(R_{ab}) \overset{p}\rightarrow SL_2(R_{ab}) \hspace{1cm} $$
By this sequence, it is clear that $[\beta(T)] \in G'(R_{ab})$. Since $G'(R)$ is an algebraic analogue of $\text{Cont}(X(\mathbb{R}), \mathbb{C})$, we can think $[\beta(T)]$ as a map from $U_1\cap U_2$ to $\mathbb{C}$. As in the case of $\phi$, we want to write $\beta(T) = \beta_1(T) \beta_2^{-1}(T)$, where $\beta_1(T)\in SL_2(R_{a}[T])$ and $\beta_2(T)\in SL_2(R_{b}[T])$. This one we will get by Quillen splitting. Further since $\beta(1) = I_2$, we patch $\beta_1(1) = \beta_2(1)$ to get an element $\Gamma(R)$. This is similar to the use of the exponential map in the case of the definition  of $\phi$. 

We would like to remark that many of the ideas in the paper are due to Krusemeyer \cite{MK}, and Karoubi and Villamayor \cite{KV}. In particular, the sequence $(5'')$ is mentioned on page 21 of \cite{MK} in a more general context.

For a ring $R$, the main aim of this paper is to explore the group $\pi_1 (SL_{2}(R))$ in details analogous to the group $H^0 (X, \mathbb{Z}) \cong \mathbb{Z}^r$, where $r$ is the number of connected components of $X$. Throughout the paper a ring means a commutative ring with identity and $SL_{2}(R)$ denotes the set of all $2\times 2$ determinant $1$ matrices with entries from $R$.
\section{Preliminaries}
We start this section with the definitions of the groups $\pi_1 (SL_2 (R))$ and $\Gamma(R)$ from \cite{RSS}. 
\begin{definition}
Let $R$ be a ring and $L$ be the set of all loops in $SL_{2}(R)$ starting and ending at the identity matrix $I_2$, that is, $L =\{\alpha(T)\in SL_{2}(R[T])\mid \alpha(0) = \alpha(1) = I_2\}$. We say that two loops $\alpha(T), \beta(T) \in L$ are equivalent (that is, written as $\alpha(T)\sim \beta(T)$) if there exists $\gamma(T, S)\in SL_{2}(R[T, S])$ such that $\gamma(T, 0) = \alpha(T), \gamma(T, 1) = \beta(T)$ and $\gamma(0,S) = \gamma(1, S) = I_2$. We call $\gamma(T, S)$ as a homotopy between $\alpha(T)$ and $\beta(T)$.

For a ring $R$, the set of all equivalence classes of loops based on $I_2$ forms an abelian group with respect to  the binary operation `$*$' defined by $[\alpha(T)]*[\beta(T)] = [\alpha(T)\beta(T)]$. We denote it by $\pi_1 (SL_{2}(R))$ and call it the algebraic fundamental group of $R$. 
\end{definition}
Before starting the definition of the group $\Gamma(R)$, we define unimodular rows over a ring.
\begin{definition}\label{unimodular}
Let $R$ be a ring and $(a_1,a_2,\ldots,a_n)\in R^n$. We call the row $(a_1,a_2,\ldots,a_n)$ is unimodular of length $n$ if there exists $b_1,b_2,\ldots,b_n\in R$ such that $a_1b_1 + \cdots + a_nb_n = 1$. 
\end{definition}
\begin{definition}\label{completable}
Let $R$ be a ring and $(a_1,a_2,\ldots,a_n)$ be a unimodular row over $R$. We say the row $(a_1,a_2,\ldots,a_n)$ is completable (elementary completable) if there exists a matrix in $GL_n(R)$ $(E_n(R))$ whose first row or column is $(a_1,a_2,\ldots,a_n)$.
\end{definition}
\begin{definition}\label{def: G(A)}
We say that two unimodular rows $(a, b),~ (c, d)$ over $R$ are equivalent, written as $(a, b)\sim (c, d)$, if one (and hence both) of the following equivalent conditions holds:
\begin{enumerate}
\item there exists $(f_{1} (T), f_{2}(T)) \in \Um_2 (R[T])$ such that $(f_{1} (0), f_{2}(0)) = (a, b)$ and $(f_{1} (1), f_{2}(1)) = (c, d)$.
\item there exists a matrix $\alpha \in SL_2 (R)$ which is connected to the identity matrix (that is, there exists a matrix $\beta(T) \in SL_2 (R[T])$ such that $\beta(0) = I_2$ and $\beta(1) = \alpha$) such that
$\alpha \begin{pmatrix}
a \\ b
\end{pmatrix} = \begin{pmatrix}
c \\ d
\end{pmatrix}$.
\end{enumerate}

It is not hard to check that the relation $\sim$ is an equivalence relation. We denote  the equivalence class of $(a,b)$ by $[a,b]$. Let $\Gamma(R)$ be the set of all equivalence classes of unimodular rows given by the equivalence relation $\sim$ as above. Define a product $*$ in $\Gamma(R)$ as follows:

Let $(a, b),~ (c, d) \in \Um_2 (R)$. Complete these to $SL_2(R)$ matrices $\sigma = \begin{pmatrix}
a & e \\ b & f
\end{pmatrix} $ and $\tau = \begin{pmatrix}
c & g \\ d & h
\end{pmatrix}$. We define the product of two elements $[a,b], [c,d]\in \Gamma(R)$ as follows: $$[a, b] * [c, d] = [\fr\;\co\;\of~ \sigma\tau] = [ac+de, bc+df].$$

Then $(\Gamma(R), *)$ is a group and the identity element of the group $\Gamma(R)$ is $[1, 0]$ (for details see \cite{RSS}). We called it as algebraic cohomotopy group of $R$.
\end{definition}
Now we would like to mention the following two results which are useful for the article.
\begin{lemma}[\cite{DQ}](Quillen's Splitting)\label{Quillen}
Let $R$ be a ring and $s,t\in R$ be such that $sR + tR = R$. Suppose there exists $\sigma(X)\in GL_n(R_{st}[X])$ with the property that $\sigma(0) = I_n$. Then there exist $\psi_1(X)\in GL_n(R_{s}[X])$ with $\psi_1(0) = I_n$ and $\psi_2(X)\in GL_n(R_{t}[X])$ with $\psi_2(0) = I_n$ such that $\sigma(X) = (\psi_1(X))_t(\psi_2(X))_s$.
\end{lemma}
\begin{theorem}[\cite{RSS}]
Let $R$ be a ring and $(a, b)\in Um_2(R)$. Then we have the following algebraic analogue of the Mayer Vietoris sequence
$\pi_1 (SL_{2}(R))\longrightarrow \pi_1 (SL_{2}(R_a)) \oplus \pi_1 (SL_{2}(R_b)) \longrightarrow\pi_1(SL_2(R_{ab}))\overset{\delta}{\rightarrow} \Gamma(R)\rightarrow \Gamma(R_a)\oplus \Gamma(R_b) \longrightarrow \Gamma(R_{ab})$.
\end{theorem}

\section{Algebraic fundamental group $\pi_1 (SL_2 (R))$}
Let $R$ and $S$ be two rings and $\phi : R \rightarrow S$ be a ring homomorphism with $\phi(1) =1$. Then we can extend $\phi$ to a ring homomorphism $\phi(T): R[T] \rightarrow S[T]$, where $\phi(T)$ is defined by $\phi(T) (\sum a_iT^i) =  \sum \phi(a_i)T^i$. We denote image of $f(T) \in R[T]$ under $\phi(T)$ by $\phi (f)(T)$. 

Let $\alpha(T) =\begin{pmatrix}
f_1 (T) & f_3 (T)  \\
f_2 (T) & f_4 (T) 
\end{pmatrix} \in SL_2 (R[T])$. Then $\begin{pmatrix}
\phi (f_1) (T) & \phi (f_3)(T)  \\
\phi (f_2) (T) & \phi (f_4) (T) 
\end{pmatrix} \in SL_2 (S[T])$. We denote it by $\phi (\alpha)(T)$. Therefore, we have a group homomorphism $\pi_1 (\phi): \pi_1 (SL_2 (R)) \rightarrow \pi_1 (SL_2 (S))$ defined by $\pi_1 (\phi)([\alpha(T)]) = [\phi (\alpha)(T)]$. 

Suppose $R$ is the coordinate ring of a real affine variety with real points $X(\mathbb{R})$. Then the real points of the ring $R[X]$ is $X(\mathbb{R})\times \mathbb{R}$. Since $X(\mathbb{R})$ and $X(\mathbb{R})\times \mathbb{R}$ have the same number of connected components, $H^0(X(\mathbb{R})) \cong H^0(X(\mathbb{R})\times \mathbb{R})$. The following theorem shows that an analogous result is also true for algebraic fundamental groups.

\begin{theorem}\label{polynomial}Let $R$ be a ring. Then
$\pi_1 (SL_2 (R[X])) \cong \pi_1 (SL_2 (R))$. Moreover, $\pi_1 (SL_2 (R[X_1, X_2, \cdots, X_n])) \cong \pi_1 (SL_2 (R))$, for all $n\in \mathbb{N}$.
\end{theorem}
\begin{proof}
Consider the ring homomorphism $\phi : R[X] \rightarrow R$ defined by $\phi(f(X)) = f(0)$. So, we have a group homomorphism $\pi_1 (\phi): \pi_1 (SL_2 (R[X])) \rightarrow \pi_1 (SL_2 (R))$ defined by $\pi_1 (\phi)([\alpha(X)(T)]) = [\alpha(0)(T)]$.

Let $[\alpha(T)]\in \pi_1 (SL_2 (R))$. Then, for $\beta(X)(T) = \alpha(T)$, we have $\pi_1 (\phi)([\beta(X)(T)]) = [\beta(0)(T)]= [\alpha(T)]$. Hence, $\pi_1 (\phi)$ is surjective.

Suppose $\pi_1 (\phi)([\alpha(X)(T)]) = \pi_1 (\phi)([\beta(X)(T)])$. Thus $[\alpha(0)(T)] = [\beta(0)(T)]$ in $\pi_1 (SL_2 (R))$. So, there exists $\theta(T, W) \in SL_2 (R[T,W])$ such that $\theta(T, 0)=\alpha(0)(T), \\\theta(T, 1)=\beta(0)(T)$ and $\theta(0, W) = \theta(1, W)= I_2$. 

Consider $M(X)(T)(W) = \alpha(X(1-W))(T)\theta(T, 1-W)^{-1}\beta(XW)(T)$. Then 

$$M(X)(T)(0) = \alpha(X)(T)\theta(T, 1)^{-1}\beta(0)(T) = \alpha(X)(T)$$
$$M(X)(T)(1) = \alpha(0)(T)\theta(T, 0)^{-1}\beta(X)(T) = \beta(X)(T)$$
Since $\alpha(X)(0) = \beta(X)(0) = I_2$ and $\alpha(X)(1) = \beta(X)(1) = I_2$, we have
$$M(X)(0)(W) = \alpha(X(1-W))(0)\theta(0, 1-W)^{-1}\beta(XW)(0) = I_2$$
$$M(X)(1)(W) = \alpha(X(1-W))(1)\theta(1, 1-W)^{-1}\beta(XW)(1) = I_2.$$
This shows that $[\alpha(X)(T)] = [\beta(X)(T)]$ and hence, $\pi_1 (\phi)$ is injective. Hence, $\pi_1 (\phi)$ is an isomorphism, that is, $\pi_1 (SL_2 (R[X])) \cong \pi_1 (SL_2 (R))$. By induction, we have 
$\pi_1 (SL_2 (R[X_1, X_2, \cdots, X_n])) \cong \pi_1 (SL_2 (R))$, for all $n\in \mathbb{N}$.
\end{proof}

\begin{remark}
Let $R$ be a ring such that $\pi_1 (SL_2 (R)) \cong \mathbb{Z}^r$, for some non-negative integer $r$. Then by Theorem \ref{polynomial}, $\pi_1 (SL_2 (R[X_1, X_2, \cdots, X_n])) \cong \mathbb{Z}^r$. This looks like the freeness of a projective module over a polynomial ring over a field (\cite{DQ} and \cite{SU}). That is why we entitled the paper by Quillen-Suslin Theory.
\end{remark}

It is known that the real affine varieties of rings $R$ and $R/I$ are same, where $I$ is the nil radical of $R$. Suppose $X(\mathbb{R})$ is the real points of the real affine variety of $R$ as well as of $R/I$. Then by the motivation of $H^0(X(\mathbb{R}))$, there is a natural question that whether the groups $\pi_1 (SL_2 (R))$ and $\pi_1 (SL_2 (R/I))$ are isomorphic? We show that it is true. But to prove this result, we need the following two propositions.

\begin{proposition}\label{elementary connection}
Let $R$ be a ring and $I$ be the nil radical ideal of $R$. Consider $\alpha =\begin{pmatrix}
1+a_1 & a_3  \\
a_2 & 1+a_4 
\end{pmatrix}\in SL_2 (R)$, where $a_i \in I$, for every $1\leq i \leq 4$ i.e., $\alpha =I_2$ modulo $I$. Then $\alpha = E_{12}(c_1)E_{21}(c_2)E_{12}(-1+c_3)E_{21}(c_4)E_{12}(1)E_{21}(c_5)$, for some $c_i \in I$, for every $1\leq i \leq 5$. Moreover, we can find  a matrix $\beta(X) \in SL_2 (R[X])$ such that $\beta(0) = I_2,  \beta(1) = \alpha$ and $\beta(X) = I_2$ modulo $I[X]$.
\end{proposition}
\begin{proof} It is easy to see that
$$\alpha = \begin{pmatrix}
1 & (1+a_4)^{-1}a_2  \\
0& 1 
\end{pmatrix}\begin{pmatrix}
1 & 0  \\
x^{-1}a_3& 1 
\end{pmatrix}\begin{pmatrix}
1 & -x  \\
0& 1 
\end{pmatrix}\begin{pmatrix}
1 & 0  \\
a_4& 1 
\end{pmatrix}\begin{pmatrix}
1 & 1  \\
0& 1 
\end{pmatrix}\begin{pmatrix}
1 & 0  \\
x-1& 1 
\end{pmatrix},$$ where $x= (1+a_1)- (1+a_4)^{-1}a_2a_3$. 

Thus, $$\alpha = E_{12}(c_1)E_{21}(c_2)E_{12}(-1+c_3)E_{21}(c_4)E_{12}(1)E_{21}(c_5),$$ for some $c_i \in I$, $1\leq i \leq 5$.

Take $\beta(X) = E_{12}(c_1X)E_{21}(c_2X)E_{12}(-1+c_3X)E_{21}(c_4X)E_{12}(1)E_{21}(c_5X)$. Then $\beta(0) = I_2,  \beta(1) = \alpha$ and $\overline{\beta(X)} = \overline{I_2}$.
\end{proof}

\begin{proposition}\label{loop connection}
Let $R$ be a ring and $I$ be the nil radical ideal of $R$. Let $\alpha(X)\in SL_2 (R[X])$ such that $\alpha(0) = I_2 = \alpha(1)$. Suppose $\overline{\alpha(X)} = \overline{I_2}$ i.e., $\alpha(X) = I_2$ modulo $I[X]$. Then there exists a matrix $\beta(X, T) \in  SL_2 (R[X, T])$ such that $\beta(X, 0) = I_2, \beta(X, 1) = \alpha(X) $, $\beta(0, T) = I_2 =\beta(1, T)$ and $\overline{\beta(X, T)} = \overline{I_2}$.
\end{proposition}
\begin{proof}
Since $\alpha(0) = I_2 = \alpha(1)$,  $\alpha(X) =\begin{pmatrix}
1+X(X-1)f_1(X) & X(X-1)f_3(X)  \\
X(X-1)f_2(X)& 1+X(X-1)f_4(X) 
\end{pmatrix}$, for some $f_i(X)\in R[X]$, $1\leq i \leq 4$. Further since  $\overline{\alpha(X)} = \overline{I_2}$,  $X(X-1)f_i(X) \in I[X]$, for all $1\leq i \leq 4$. Now by Proposition \ref{elementary connection},  $\alpha(X) = E_{12}(h_1(X))E_{21}(h_2(X))E_{12}(-1+h_3(X))E_{21}(h_4(X))E_{12}(1)E_{21}(h_5(X))$, for some $h_i(X)  = X(X-1)g_i(X) \in I[X]$, for every $1\leq i \leq 5$, where $g_i(X)\in R[X]$. 

Consider $\beta(X, T) = E_{12}(h_1(X)T)E_{21}(h_2(X)T)E_{12}(-1+h_3(X)T)E_{21}(h_4(X)T)E_{12}$ $(1)E_{21}(h_5(X)T)$. It is clear that $\beta(X, 0) = I_2, \beta(X, 1)$ $= \alpha(X) $ and $\overline{\beta(X, T)} = \overline{I_2}$. Since $h_i(0)  = 0 = h_i(1)$ for every $1\leq i \leq 5$, $\beta(0, T) = I_2 =\beta(1, T)$.
\end{proof}

\begin{theorem}
Let $R$ be a ring and $I$ be the nil radical ideal of $R$. Then $\pi_1 (SL_2 (R)) \cong \pi_1 (SL_2 (R/I))$.
\end{theorem}

\begin{proof}
Consider the quotient ring homomorphism $\nu : R \rightarrow R/I$. Then we have a group homomorphism  $\pi_1 (\nu): \pi_1 (SL_2 (R)) \rightarrow \pi_1 (SL_2 (R/I))$ defined by $\pi_1 (\nu)([\alpha(X)]) = [\overline{\alpha(X)} ]$.

\textbf{Claim 1:} $\pi_1 (\nu)$ is surjective. 

Let $[\beta(X)] \in \pi_1 (SL_2 (R/I))$. Then $\beta(0) = \beta(1) = \overline{I_2}$. Suppose $\beta(X) = \begin{pmatrix}
F_1 (X) & F_3 (X)  \\
F_2 (X) & F_4 (X) \end{pmatrix}$, where $F_i (X) \in SL_2 (R/I[X])$, for every $1 \leq i \leq 4$. Then $F_i (X)  = f_i (X) + \lambda_i (X) $, for every $1 \leq i \leq 4$, where $f_i (X)  \in R[X]$ and $\lambda_i (X)  \in I[X]$.

Since $F_1(X)F_4(X)-F_2(X)F_3(X) = \bar{1}$, $f_1(X)f_4(X)-f_2(X)f_3(X) = 1+ \lambda(X)$ for some $\lambda (X)  \in I[X]$. Clearly $1+ \lambda(X)$ is a unit element in $R[X]$ and $(1+ \lambda(X))^{-1} = 1+ \lambda'(X)$ for some $\lambda'(X)  \in I[X]$. 

Let $\alpha(X) = \begin{pmatrix}
(1+ \lambda'(X))f_1 (X) & f_3 (X)  \\
(1+ \lambda'(X))f_2 (X) & f_4 (X) \end{pmatrix}$. Then $\alpha(X) \in SL_2 (R[X])$ and $\overline{\alpha(X)}= \overline{\beta(X)}$. Since $\overline{\beta(0)} = \overline{\beta(1)} =  \overline{I_2}$, $\overline{\alpha(0)} = \overline{\alpha(1)} =  \overline{I_2}$. So, by Proposition \ref{elementary connection}, we can find matrices $\theta_1 (X)$  and $\theta_2 (X)$  in $ SL_2 (R[X])$  such that 
\begin{align*}
\theta_1(0) =I_2, \theta_1(1) = \alpha(0),  \overline{\theta_1 (X)} =  \overline{I_2}\\
\theta_2(0) =I_2, \theta_2(1) = \alpha(1),  \overline{\theta_2 (X)} =  \overline{I_2}
\end{align*}
Consider $\gamma(X) =\theta_1 (1-X)^{-1} \alpha(X) \theta_2 (X)^{-1}$. Then $\overline{\gamma(X)}= \overline{\alpha(X)}$ and $\gamma(0) = \gamma(1) = I_2$. Therefore $\pi_1 (\nu)([\gamma(X)]) = [\overline{\gamma(X)} ]=  [\overline{\alpha(X)}]= [\overline{\beta(X)}] $.

\textbf{Claim 2:} $\pi_1 (\nu)$ is injective. 

Let $[ \alpha(X)] \in \ker (\pi_1 (\nu))$ i.e,, $[\overline{\alpha(X)}] = [ \overline{I_2}]$. Then there exists $\beta'(X, T)\in SL_2 (R/I[X, T])$ such that $\beta'(X, 0) = \overline{\alpha(X)}, \beta'(X, 1) =  \overline{I_2}$ and $\beta'(0, T) = \beta'(1, T) =  \overline{I_2}$.  

Since $\beta'(0, T) = \beta'(1, T) =  \overline{I_2}$,  $[\beta'(X, T)]\in \pi_1 (SL_2 (R/I [T]))$. Since $\pi_1 (\nu): \pi_1 (SL_2 (R[T])) \rightarrow \pi_1 (SL_2 (R/I [T]))$ is surjective, there exists a $\beta(X, T) \in SL_2 (R[X, T])$  such that $\overline{\beta(X, T)}= \overline{\beta'(X, T)}$ and $\beta(0, T) = \beta(1, T) =  I_2$.

Since $\overline{\beta(X, 0)} = \overline{\alpha(X)}$, $\overline{\beta(X, 0)\alpha(X)^{-1}} = \overline{I_2}$. Also we have $\beta(0, 0)\alpha(0)^{-1}=\beta(1, 0)\alpha(1)^{-1} = I_2$. By  Proposition \ref{loop connection}, we can find a matrix $\gamma_1 (X, T) \in SL_2 (R[X, T])$ such that $\gamma_1 (X, 0) = I_2, \gamma_1 (X, 1)  = \beta(X, 0)\alpha(X)^{-1}$ and $\gamma_1 (0, T) = I_2= \gamma_1 (1, T)$.

Since $ \overline{\beta(X, 1)} = \overline{I_2}$ and $\beta(0, 1)=\beta(1, 1) = I_2$, by  Proposition \ref{loop connection}, we can find a matrix $\gamma_2 (X, T) \in SL_2 (R[X, T])$ such that $\gamma_2 (X, 0) = I_2, \gamma_2 (X, 1)  = \beta(X, 1)$ and $\gamma_2 (0, T) = I_2= \gamma_2 (1, T)$.

Consider $\Gamma(X, T) = \gamma_1 (X, 1-T)^{-1}\beta(X, T)\gamma_2 (X, T)^{-1}$. Then $\Gamma(X,0) = \alpha(X)$, $\Gamma(X,1) = I_2$ and $\Gamma(0, T) = \Gamma(1, T) =  I_2$. Hence $[\alpha(X)] = [I_2]$. Therefore $\pi_1 (\nu)$ is injective. 
\end{proof}

Since projective variety $X(\mathbb{R})$ corresponding to a graded ring $R = \mathbb{R} \oplus R_1 \oplus \cdots $ is a cone and cone is connected, $H^{0}(X(\mathbb{R}), \mathbb{Z}) \cong H^{0}(\mathbb{R}, \mathbb{Z})$. The following theorem is an algebraic analogue of this and we will use the Swan-Weibel Homotopy Trick (\cite{LAM}, page 183) to prove it.
\begin{theorem}
Let $R = R_0 \oplus R_1 \oplus \cdots $ be a positively graded ring. Then $\pi_1 (SL_2 (R_0)) \cong \pi_1 (SL_2 (R))$.
\end{theorem}
\begin{proof}
Consider the inclusion map $j :R_0 \rightarrow R$. Then the map $\pi_1 (j)  : \pi_1 (SL_2 (R_0)) \rightarrow \pi_1 (SL_2 (R))$ is a group homomorphism. Let 
$[\alpha_0 (X)] \in \ker (\pi_1 (j))$ i.e., $[\alpha_0 (X)] = I_2$. Then there exists $\beta(X, T) = \begin{pmatrix}
f_1 (X, T) & f_3 (X, T)  \\
f_2 (X, T) & f_4 (X, T) \end{pmatrix}\in SL_2(R[X, T])$ such that $\beta(X, 0) = \alpha_0 (X), \beta(X, 1) = I_2$ and $\beta(0,T) = \beta(1, T) = I_2$. Take $\pi(\beta(X, T)) = \begin{pmatrix}
\pi(f_1 (X, T)) & \pi(f_3 (X, T) ) \\
\pi(f_2 (X, T)) & \pi(f_4 (X, T)) \end{pmatrix}$, where $\pi: R[X, T] \rightarrow R_0[X, T]$ is the projection homomorphism. Then $\pi(\beta(X, T)) \in SL_2(R_0[X, T])$ with $\pi(\beta(X, 0)) = \alpha_0 (X)$, $\pi(\beta(X, 1)) = I_2$ and $\pi(\beta(0,T) )= \pi(\beta(1, T)) = I_2$. This show that $\ker (\pi_1 (j)) = \{[I_2]\}$. Hence $\pi_1 (j)$ is injective. 

Let $[\beta(X)] = \begin{pmatrix}
1+X(X-1)f_1 (X) & X(X-1)f_3 (X)  \\
X(X-1)f_2 (X) & 1+X(X-1)f_4 (X) \end{pmatrix}\in \pi_1 (SL_2 (R))$. Consider the Swan-Weibel ring homomorphism $h(T): R \rightarrow R[T]$ defined by $h(T)(a_0+a_1+\cdots+a_n) = a_0+a_1T+\cdots+a_nT^n (n \geq 0, a_i \in R_i)$. It is clear that $h(0) = a_0$ and $h(1) = a_0+a_1+\cdots+a_n$.

Suppose, for $i = 1, 2, 3, 4$, $f_i(X) = a^{i}_0 + a^{i}_1X+ \cdots + a^{i}_{r_i} X^{r_i}$, where $a^{i}_j \in A$ for every $0 \leq j \leq r_i$.
Take $F_i (X, T ) = h(T)(a^{i}_0) + h(T)(a^{i}_1)X+ \cdots + h(T)(a^{i}_{r_i}) X^{r_i}$, for all $i = 1, 2, 3, 4$. Clearly,  for $i = 1, 2, 3, 4$, $F_i (X, 0 ) = h(0)(a^{i}_0) + h(0)(a^{i}_1)X+ \cdots + h(0)(a^{i}_{r_i}) X^{r_i} \in R_0 [X]$ since $h(0)(a^{i}_j) \in R_0$ for every $0 \leq j \leq r_i$ and $F_i (X, 1) = h(1)(a^{i}_0) + h(1)(a^{i}_1)X+ \cdots + h(1)(a^{i}_{r_i}) X^{r_i} = f_i(X) $.

Now, consider $\beta_0(X)= \begin{pmatrix}
1+X(X-1)F_1(X, 0 )& X(X-1)F_3(X, 0 )  \\
X(X-1)F_2(X, 0 ) & 1+X(X-1)F_4 (X, 0 )\end{pmatrix}$. Since $h(T)$ is a ring homomorphism, $[\beta_0(X)] \in  \pi_1 (SL_2 (R_0))$. 

Take $\gamma(X, T)= \begin{pmatrix}
1+X(X-1)F_1(X, T )& X(X-1)F_3(X, T)  \\
X(X-1)F_2(X, T) & 1+X(X-1)F_4 (X, T)\end{pmatrix}$. Then $\gamma(X, 0)$ $= \beta_0(X), \gamma(X, 1)= \beta(X)$  and $\gamma(0, T)= I_2 = \gamma(1, T)$. Hence, $[\beta_0(X)] = [\beta(X)]$ in $\pi_1 (SL_2 (R))$. Therefore, $\pi_1 (j)([\beta_0(X)]) = [\beta(X)]$ i.e., $\pi_1 (j)$ is surjective. 
\end{proof}
The next result is an algebraic analogue of $H^0(Z) \cong H^0(X) \oplus H^0(Y)$, where $X$ and $Y$ are disjoint open subsets of a topological space $Z$ with $Z = X \cup Y$.
\begin{theorem} Let $R$ and $S$ be rings. Then
$\pi_1 (SL_2 (R\times S)) \cong \pi_1 (SL_2 (R))\oplus \pi_1 (SL_2 (S))$.
\end{theorem}

\begin{proof}
Since $ SL_2 (R\times S) \cong  SL_2 (R)\oplus  SL_2 (S) $, $ SL_2 ((R\times S)[X]) \cong  SL_2 (R[X])\oplus  SL_2 (S[X]) $. Therefore $\pi_1 (SL_2 (R\times S)) \cong \pi_1 (SL_2 (R))\oplus \pi_1 (SL_2 (S))$.
\end{proof}

The following theorem is a version of Horrocks Theorem \cite{GH} for $\pi_1 (SL_2 (R))$.
\begin{theorem}
Let $f(X)$ be a monic polynomial in $R = A[X]$. Then we have an injective homomorphism from $\pi_1 (SL_2 (R))$ to $\pi_1 (SL_2 (R_f))$.
\end{theorem}
\begin{proof}
Let $[\alpha(X)(T)] \in \pi_1 (SL_2 (R))$ such that $[\alpha(X)(T)] = [I_2]$ in $\pi_1 (SL_2 (R_f))$. Consider $\gamma(X, T, S) = \alpha((X-1)S +1)(T)\in SL_2 (A[X, T, S])$. Then, $\gamma(X, T, 0)= \alpha(1)(T), \gamma(X, T, 1)= \alpha(X)(T)$ and $\gamma(X, 0, S) = \gamma(X, 1, S) = I_2$. Therefore, 
$[\alpha(X)(T)]$ $=[\alpha(1)(T)]$ in $\pi_1 (SL_2 (R))$. So, to prove the theorem it is enough to show that $[\alpha(1)(T)] =  [I_2]$ in $\pi_1 (SL_2 (R))$. 

Since $[\alpha(X)(T)] = [I_2]$ in $\pi_1 (SL_2 (R_f))$, $[\alpha(X)(T)] = [I_2]$ in $\pi_1 (SL_2 (R_{Xf}))$. Let $Y = X^{-1}$ and $g(Y) = X^{-\text{deg}(f)}f(X)$. Since the constant term of $g$ is unit, $(Y, g(Y))$ is unimodular over $A[Y]$. Since $A[Y]_{Yg} = A[X, X^{-1}]_f = A[X]_{Xf}$, $[\alpha(Y^{-1})(T)] = [I_2]$ in $\pi_1 (SL_2 (A[Y]_{Yg}))$. Then there exists $\sigma(Y, T, S) \in SL_2 (A[Y]_{Yg}[T, S])$ such that $\sigma(Y, T, 0) = I_2, \sigma(Y, T, 1) = \alpha(Y^{-1})(T)$ and $\sigma(Y, 0, S) = I_2 = \sigma(Y, 1, S)$. By Lemma \ref{Quillen}, there exist $\sigma_1(Y, T, S) \in SL_2 (A[Y]_{g})$ and $\sigma_2(Y, T, S) \in SL_2 (A[Y]_{Y})$ such that $\sigma(Y, T, S) = \sigma_1(Y, T, S) \sigma_2^{-1}(Y, T, S)$. Therefore,  $\sigma_1(Y, T, 0)=\sigma_2(Y, T, 0) = I_2, \sigma_1(Y, T, 1) = \alpha(Y^{-1})(T)\sigma_2(Y, T, 1)$ and $\sigma_1(Y, 0, S) = \sigma_1(Y, 1, S)= \sigma_2(Y, 0, S) = \sigma_2(Y, 1, S) = I_2$. By patching $\sigma_1(Y, T, 1)$ and $\alpha(Y^{-1})(T)\sigma_2(Y, T, 1)$ together, we get $\delta(Y, T) \in SL_2 (A[Y, T])$ such that $\delta(Y, T) = \sigma_1(Y, T, 1)$ in  $SL_2 (A[Y]_{g})$ and $\delta(Y, T) = \alpha(Y^{-1})(T)\sigma_2(Y, T, 1)$ in  $SL_2 (A[Y]_{Y})$.

Take $\theta_1 (T, S)= \sigma_1(0, T, S)$ and $\theta_2 (T, S)= \sigma_2(1, T, S)$. Thus $[\delta(0, T)] = [\sigma_1(0, T, 1)]$ $= [I_2]$ and $[\sigma_2(1, T, 1)] = [I_2]$ in $\pi_1 (SL_2 (A))$ (so in $\pi_1 (SL_2 (R))$). Hence, we have
$[\delta(1, T)] = [\alpha(1)(T)][\sigma_2(1, T, 1)] = [\alpha(1)(T)]$ in $\pi_1 (SL_2 (A))$. We know $[\delta(0, T)] = [\delta(1, T)]$. So,  $ [\alpha(1)(T)]= [I_2]$ in $\pi_1 (SL_2 (A))$ (so in $\pi_1 (SL_2 (R))$).
\end{proof}

\section{Computation of Algebraic Fundamental group}
We start the section with a remark on the concepts of winding numbers.
\begin{remark}\label{winding} All the concepts mentioned in this remark are taken from \cite{JR}.
\begin{enumerate}
\item Let $f: I \rightarrow \mathbb{R}^2-\{0\}$ be  a loop. Then $g: I \rightarrow  S^1$ defined by $g(t) = \frac{f(t)}{||f(t)||}$ is a loop in $S^1$. Consider the standard 
covering map $p: \mathbb{R} \rightarrow S^1$. Then there exists a lifting $\tilde{g}: I \rightarrow \mathbb{R}$ of $g$ such that the following diagram 
\begin{tikzcd}
 {} & \mathbb{R} \arrow{d}{p}\\
I\arrow{ur}{\tilde{g}} \arrow{r}{g} & S^1 \\
\end{tikzcd} is commutative. Since $g$ is a loop, the difference $\tilde{g}(1) - \tilde{g}(0)$ is an integer. This integer is called the \textbf{winding number} of $f$.
\item Let $f, f': I \rightarrow \mathbb{R}^2-\{0\}$ be  two loops. We say that $f$ is free homotopic to $f'$ if there exists a continuous map $F : I \times I \rightarrow \mathbb{R}^2-\{0\}$ such that $F(t, 0)= f(t)$, $F(t, 1)= f'(t)$ and $F(0, s)= F(1, s)$, for all $s\in I$.

It is easy to see that if $f$ is free homotopic to $f'$, the \textbf{winding number} of $f$ is equal to the \textbf{winding number} of $f'$.
\item Let $f, f'$ be two loops in $S^1$, Let $\tilde{g}$ and $\tilde{g'}$ be lifting of $f$ and $f'$ respectively. Then $\tilde{g} + \tilde{g'}$ is a lifting of $f.f'$. In other words, the \textbf{winding number} of $f.f'$ is equal to the sum of the \textbf{winding number} of $f$ and the \textbf{winding number} of $f'$.
\item Consider the loop defined by $g(t) = e^{2\pi i t}$. Then $\tilde{g} (t) = t$ is a lift of $g$. Hence winding number of $g$ is $\tilde{g} (1) - \tilde{g} (0) =1$.
\end{enumerate}
\end{remark}

\begin{definition}
Let $[\alpha(T)]\in \pi_1 (SL_2 (\mathbb{R}))$. Then the first column of $\alpha(T)$ gives a loop in 
$\mathbb{R}^2-\{0\}$. We call the winding number of the loop given the first column of $\alpha(T)$ as the winding number of $[\alpha(T)]$ and we denote it by $w_{\alpha}$.
\end{definition}
The idea which we are going to use to prove the following theorem is given in \cite{BMS} (Example  4.5 (Stallings)).
\begin{theorem}\label{real number}
There exists a surjective homomorphism from $\pi_1 (SL_2 (\mathbb{R}))$ to $ \mathbb{Z}$.
\end{theorem}
\begin{proof}
Define $\eta : \pi_1 (SL_2 (\mathbb{R})) \rightarrow \mathbb{Z}$ by $\eta([\alpha(T)]) = w_{\alpha}$

\noindent \textbf{Claim:} $\eta$ is well defined.

Suppose $[\alpha(T)] = [\alpha'(T)]$. Then there exists a $\gamma(T, S)\in SL_2 (\mathbb{R}[T, S])$ such that $\gamma(T, 0)= \alpha(T)$, $\gamma(T, 1)= \alpha'(T)$ and $\gamma(1, S) = I_2 = \gamma(0, S)$. So the maps given by the first column of $\gamma(T, S)$ gives a free homotopy between the loops given by $\alpha(T)$ and $\alpha'(T)$. Hence $\eta([\alpha(T)]) = w_{\alpha} = w_{\alpha'}= \eta([\alpha'(T)])$.This shows that $\eta$ is well defined.

\noindent \textbf{Claim:} $\eta$ is a group homomorphism.

Let $[\alpha(T)], [\alpha'(T)] \in \pi_1 (SL_2 (\mathbb{R}))$. Suppose $\alpha(T) = \begin{pmatrix}
f_1 & g_1\\
f_2 & g_2
\end{pmatrix}$ and $\alpha'(T) = \begin{pmatrix}
f'_1 & g'_1\\
f'_2 & g'_2
\end{pmatrix}$. Consider the loops $F_1 = (f_1f'_1 + g_1f'_2, f_2f'_1 + g_2f'_2): I \rightarrow \mathbb{R}^2-\{0\}$ and $F_2 =(f_1f'_1 - f_2f'_2, f_2f'_1 + f_1f'_2): I \rightarrow \mathbb{R}^2-\{0\}$. Now we will show that these loops are free homotopic to each other.

Consider the continuous function $H(t, s) =(f_1f'_1 + (sg_1-(1-s)f_2)f'_2, f_2f'_1 + (sg_2+(1-s)f_1)f'_2): I \times I \rightarrow \mathbb{R}^2$. 

Suppose $H(t, s) = (0, 0)$ for some $(t, s)$. Thus we have the following equations
\begin{equation}\label{1}
f_1f'_1 + (sg_1-(1-s)f_2)f'_2 = 0
\end{equation}
\begin{equation}\label{2}
f_2f'_1 + (sg_2+(1-s)f_1)f'_2 = 0
\end{equation}
Consider $f_1 ~\text{Eq} ~(\ref{2}) -f_2~\text{Eq}~ (\ref{1}) $, then we have 
\begin{align*}
(s(f_1g_2 - f_2g_1) + (1-s)(f_1^2+f_2^2))f'_2 = 0\\
\Rightarrow (s+(1-s)(f_1^2+f_2^2))f'_2 = 0
\end{align*}
Since $(s+(1-s)(f_1^2+f_2^2))\ne 0$, $f'_2 (t) =0$. By (1) and (2), $f_1f'_1= 0 = f_2f'_1$. Since $f'_2 (t) =0$, $f'_1 (t) \ne 0$. Hence $f_1(t)= 0 = f_2(t)$ which is not possible. Hence $H(t, s) =(f_1f'_1 + (sg_1-(1-s)f_2)f'_2, f_2f'_1 + (sg_2+(1-s)f_1)f'_2): I \times I \rightarrow \mathbb{R}^2-\{0\}$. It is clear that $H(t, 0) = F_2(t)$, $H(t, 1) = F_1(t)$ and $H(0, s)= H(1, s)$ for all $s \in I$. This shows that the loops $F_1: I \rightarrow \mathbb{R}^2-\{0\}$ and $F_2: I \rightarrow \mathbb{R}^2-\{0\}$ are free homotopic to each other. Therefore, $\text{winding number of}~ F_1 = \text{winding number of}~ F_2 = \text{winding number of}~ (f_1, f_2)+ \text{winding number of}~ (f'_1, f'_2)$ (by Remark \ref{winding} (3)). 

Thus $\eta([\alpha(T)\alpha'(T)]) = \eta([\alpha(T)]) + \eta([\alpha'(T))]$, that is, $\eta$ is a group homomorphism. 

\noindent \textbf{Claim:} $\eta$ is surjective.

Consider $\gamma(t) = (1-t)+ ti$, where $i = \sqrt{-1}$. Then we have a path $\gamma: I \rightarrow \mathbb{R}^2 - \{0\}$, defined by $\gamma(t) = (1-t, t)$  joining the point $(1, 0)$ to $(0, 1)$. Now, consider the path $g$ on $S^1$ (so on $\mathbb{R}^2 - \{0\}$) joining the point $(1, 0)$ to $(0, 1)$, defined by $g(t) = e^{i\frac{\pi}{2}t} = (\cos \frac{\pi}{2}t, \sin \frac{\pi}{2}t)$.

Define a function $H: I\times I \rightarrow \mathbb{R}^2 - \{0\}$ by $H(t, s) = (1-s)\gamma(t)+ s g(t)$. Clearly, $H$ is a free homotopy between $\gamma$ and $g$.

Now, take $\gamma^2(t) = (1-2t)+ 2t(1-t)i$. So, we have a path $\gamma^2: I \rightarrow \mathbb{R}^2 - \{0\}$, defined by $\gamma^2(t) = (1-2t, 2t(1-t))$  joining the point $(1, 0)$ to $(-1, 0)$. Then $H^2 (t, s)$ (complex multiplication) gives a free homotopy between $\gamma^2$ and $g^2$.

Now, take $\gamma^3(t) = (1-t)(1-2t-2t^2)+ t(1-2t^2)i$.  So, we have a path $\gamma^3: I \rightarrow \mathbb{R}^2 - \{0\}$, defined by $\gamma^3(t) = ((1-t)(1-2t-2t^2), t(1-2t^2))$ joining the point $(1, 0)$ to $(0, -1)$. Then $H^3 (t, s)$ (complex multiplication) gives a free homotopy between $\gamma^3$ and $g^3$.

Finally, take $\gamma^4(t) = \big(1+4t(1-t)(t^2-t-1)\big)+\big(4t(1-t)(2t-1)\big)i$. So, we have a path $\gamma^4: I \rightarrow \mathbb{R}^2 - \{0\}$, defined by $\gamma^4(t) = \big(\big(1+4t(1-t)(t^2-t-1)\big), \big(4t(1-t)(2t-1)\big)$ joining the point $(1, 0)$ to $(1, 0)$. Then $H^4 (t, s)$ (complex multiplication) gives a free homotopy between $\gamma^4$ and $g^4$. In fact, the \textbf{winding number} of $\gamma^4$ and the \textbf{winding number} of $g^4$. Since by Remark \ref{winding} (4) the \textbf{winding number} of $g^4$ is $1$,  the \textbf{winding number} of $\gamma^4$ is $1$.

Consider \\
$\alpha(T) = \begin{pmatrix}
1+4T(1-T)(T^2-T-1) & T(1-T)(2T-1)(24T^2 -24T-29)\\
4T(1-T)(2T-1) & 1+4T(1-T)(24T^2 -24T-1)
\end{pmatrix}.$

Then $[\alpha(T)]\in \pi_1 (SL_2 (\mathbb{R}))$ and $w_\alpha = 1$. Hence, $\eta
([\alpha^n(T)])= n$ for all $n \in \mathbb{Z}$. This shows that $\eta$ is surjective.
\end{proof}

\noindent\textbf{Algebraic fundamental group of the coordinate ring of $S^1- \{(0, 1), (0, -1)\}$:}

Let $X= S^1$,  $U = S^1 -\{(0, 1)\}$ and  $V = S^1 -\{(0, -1)\}$. Then $U\cup V = S^1$ and  $U\cap V = S^1- \{(0, 1), (0, -1)\}$. Thus, we have the Mayer Vieoteris sequence 
$H^0 (S^1, \mathbb{Z}) \rightarrow H^0 (U, \mathbb{Z})\oplus H^0 (V, \mathbb{Z}) \rightarrow H^0(S^1- \{(0, 1), (0, -1)\}, \mathbb{Z})\rightarrow H^1 (S^1, \mathbb{Z}) \rightarrow H^1 (U, \mathbb{Z})\oplus H^1 (V, \mathbb{Z}) \rightarrow H^1(S^1- \{(0, 1), (0, -1)\}, \mathbb{Z})$. 

By the stereographic projection, $U\simeq \mathbb{R}$ and  $V \simeq \mathbb{R}$ ($\simeq$ means they are homeomorphic). Since $\mathbb{R}$ and $S^1$ are connected, $ H^0 (U, \mathbb{Z}) \cong H^0 (V, \mathbb{Z})\cong  H^0 (S^1, \mathbb{Z}) \cong \mathbb{Z}$. Further, since
$S^1- \{(0, 1), (0, -1)\}\simeq \mathbb{R} -\{a\}$ for some $a \in \mathbb{R}$  and $\mathbb{R} -\{a\}$ has two connected components, $ H^0(S^1- \{(0, 1), (0, -1)\}, \mathbb{Z}) \cong \mathbb{Z} \oplus \mathbb{Z}$.

On the other hand, since $\mathbb{R}$ is contractible and $\mathbb{R} -\{a\}$ has two contractible components, $H^1 (U, \mathbb{Z}) \cong H^1 (V, \mathbb{Z}) \cong \{1\}$ and $ H^1(S^1- \{(0, 1), (0, -1)\}, \mathbb{Z}) \cong \{1\}\oplus\{1\}$, where $1$ denotes the identity element.
 Also $H^1 (S^1, \mathbb{Z})\cong \mathbb{Z}$. Thus the above Mayer Vieoteris sequence is 
$\mathbb{Z} \rightarrow \mathbb{Z}\oplus \mathbb{Z} \rightarrow \mathbb{Z} \oplus \mathbb{Z}\rightarrow\mathbb{Z} \rightarrow \{1\}\oplus\{1\}\rightarrow\{1\}\oplus\{1\}$.  By the motivation of this sequence, we will try to give some glimpse about the algebraic fundamental group:

Consider $A = \frac{\mathbb{R}[X, Y]}{(X^2 + Y^2 -1)}$, $u = 1-y$ and $v = 1-y$, where $y$ denotes the image of $Y$ in $A$. Then $A_u = \mathbb{R}[\eta]_{1+\eta^2}$, where $\eta = \frac{x}{u}, u^{-1} = \frac{1}{2}({1+\eta^2})$ (\cite{SW}, Lemma 3.1). Similarly $A_v = \mathbb{R}[\eta]_{1+\eta^2}$, where $\eta = \frac{x}{v}, v^{-1} = \frac{1}{2}({1+\eta^2})$. This implies that $A_{uv} = \mathbb{R}[\eta]_{\eta^2(1+\eta^2)}$. Since $\mathbb{R}[\eta]_{1+\eta^2}$ is an Euclidean domain, $\Gamma(A_u) = \Gamma(A_v)= \Gamma(A_{uv}) = \{1\}$.
Now, consider algebraic Mayer Vieoteris sequence 
$\pi_1 (SL_2 (A)) \rightarrow \pi_1 (SL_2 (A_u)) \oplus \pi_1 (SL_2 (A_v)) \overset{\psi}\rightarrow \pi_1 (SL_2 (A_{uv}))\overset{\delta}\rightarrow \Gamma(A) \rightarrow \Gamma(A_u)\oplus\Gamma(A_v)\rightarrow \Gamma(A_{uv})$., where $\delta$ is the connecting map. Since $\Gamma(A_u) = \Gamma(A_v)= \Gamma(A_{uv}) = \{1\}$, $\Gamma$ is surjective. Hence we have a short exact sequence $$1\rightarrow \ker{\delta}(= Im(\psi))\rightarrow \pi_1 (SL_2 (A_{uv}))\overset{\delta}\rightarrow \Gamma(A) \rightarrow 1 \hspace{2cm} (7)$$

Since there exists a surjective homomorphism $f: \Gamma(A) \rightarrow \mathbb{Z}$ defined by $f([a, b]) = \text{deg}([a, b])$, $f\circ \delta:  \pi_1 (SL_2 (A_{uv})) \rightarrow \mathbb{Z}$ is a surjective map. Hence, we a short exact sequence $$1\rightarrow \ker{(f\circ \delta)}\rightarrow \pi_1 (SL_2 (A_{uv}))\overset{f\circ \delta}\rightarrow \mathbb{Z}\rightarrow 1 \hspace{2cm} (8)$$

Since $\mathbb{Z}$ is a free abelian group, the exact sequence (8) is split. That is, $$ \pi_1 (SL_2 (A_{uv})) \cong   \ker{(f\circ \delta)} \oplus  \mathbb{Z}$$
It is clear that $\ker{\delta} \subseteq \ker{(f\circ \delta)}$, $\mathbb{Z} \oplus \ker{\delta} \subseteq  \pi_1 (SL_2 (A_{uv})) $.

 Now, it is easy to observe that we have a surjective group homomorphism from $\ker{\delta} \rightarrow \pi_1 (SL_2 (\mathbb{R}))$. Since by Theorem \ref{real number}, we have a surjective group homomorphism  $\pi_1 (SL_2 (\mathbb{R}))\rightarrow \mathbb{Z}$, we have a surjective group homomorphism from $\ker{\delta} \rightarrow \mathbb{Z}$. Since $\mathbb{Z}$ is a free abelian group, $\mathbb{Z} \subseteq \ker{\delta}$. Therefore $\mathbb{Z} \oplus \mathbb{Z} \subseteq  \pi_1 (SL_2 (A_{uv})) $.

\begin{remark}
For a ring $R$, as we defined the group $\pi_1 (SL_2 (R))$, we can also define $\pi_1 (SL_n (R))$ for any $n \geq 3 $. Since the fundamental group of $SL_n(\mathbb{R})$ is $\mathbb{Z}_2$ for $n\geq 3 $, we can say that $\pi_1 (SL_n (R))$ is an algebraic analogue of $\text{Cont}(X(\mathbb{R}), \mathbb{Z}_2)$, where $R$ is the coordinate ring of a real variety $X$ with real points $X(\mathbb{R})$. Therefore, it is easy to develop the theory for $\pi_1 (SL_n (R))$ in a similar pattern of $\pi_1 (SL_2 (R))$.
\end{remark}
We will end this article by proposing the following problems:

\textbf{Proposed problems:} 
\begin{enumerate}
\item We know that $ H^0(S^1- \{(0, 1), (0, -1)\}, \mathbb{Z}) \cong \mathbb{Z} \oplus \mathbb{Z}$ and $\pi_1 (SL_2 (A_{uv}))$ is an algebra analogue of $ H^0(S^1- \{(0, 1), (0, -1)\}, \mathbb{Z})$, where $A = \frac{\mathbb{R}[X, Y]}{(X^2 + Y^2 -1)}$, $u = 1-y$ and $v = 1-y$, where $y$ denotes the image of $Y$ in $A$. In Section $(4)$, we have seen that $\mathbb{Z} \oplus \ker{\Gamma} \subseteq  \pi_1 (SL_2 (A_{uv})) $. Is it true that $\mathbb{Z} \oplus \mathbb{Z} \cong  \pi_1 (SL_2 (A_{uv}))$?
\item $\pi_1 (SL_2 (\mathbb{R})) \cong \mathbb{Z}$.
\item $\pi_1 \bigg(SL_2 \bigg(\frac{\mathbb{R}[X, Y]}{(Y -X^2)}\bigg)\bigg) \cong \mathbb{Z}$;
\item $\pi_1 \bigg(SL_2 \bigg(\frac{\mathbb{R}[X, Y]}{(Y^2 -X^3)}\bigg)\bigg) \cong \mathbb{Z}$;
\item $\pi_1 \bigg(SL_2 \bigg(\frac{\mathbb{R}[X, Y]}{(Y^2 -X^2 (X+1))}\bigg)\bigg) \cong \mathbb{Z}$;
\item $\pi_1 \bigg(SL_2 \bigg(\frac{\mathbb{R}[X, Y]}{(XY-1)}\bigg)\bigg) \cong\mathbb{Z} \oplus \mathbb{Z}$;
\item $\pi_1 \bigg(SL_2 \bigg(\frac{\mathbb{R}[X, Y]}{(Y^2 -X (X^2-1))}\bigg)\bigg) \cong  \mathbb{Z} \oplus \mathbb{Z}$;
\end{enumerate}

\end{document}